\documentclass[11pt]{article}
\usepackage{mathrsfs}
\usepackage{bbm,color}
\usepackage{amsfonts}
\usepackage{amsmath,amssymb}
\usepackage{amsmath}
\usepackage{amssymb}
\usepackage{graphicx}

\newtheorem{theorem}{Theorem}[section]

\newtheorem{definition}[theorem]{Definition}

\newtheorem{lemma}[theorem]{Lemma}
\newtheorem{proposition}[theorem]{Proposition}

\newenvironment{proof}[1][Proof]{\noindent \textbf{#1.} }{\  $\Box$}
\numberwithin{equation}{section}

\begin{document}

\title{\textbf{Anticipated mean-field backward stochastic differential equations with jumps}}
\author{Tao HAO\thanks{%
School of Statistics, Shandong University of Finance and Economics, Jinan 250014, P. R. China.
haotao2012@hotmail.com. Research supported by National Natural Science Foundation of China
(Grant Nos. 71671104,11871309,11801315,71803097), the Ministry of Education of Humanities and Social Science Project (Grant No. 16YJA910003),
Natural Science Foundation of Shandong Province (No. ZR2018QA001),
A Project of Shandong  Province Higher Educational Science and Technology Program (Grant Nos. J17KA162, J17KA163),
and Incubation Group Project of Financial Statistics and Risk Management of SDUFE. }
 }
\maketitle
\date{}

\begin{abstract}
In this paper we prove the existence and uniqueness theorem, comparison theorem of a class of anticipated
mean-field backward stochastic differential equations with jumps.
\end{abstract}


\textbf{MSC-classification}: 60H10\\\\


\section{{\protect \large {Introduction}}}

Stochastic delay differential equations (SDDEs) can be met frequently in the fields of Finance, Economics and Physics. Recently,
stochastic optimal control problems for SDDEs have attracted an increasing attention. We refer to
Chen, Huang \cite{CH},
Chen, Wu \cite{CW},
Elsanosi, $\phi$ksendal, Sulem \cite{EOS},
Guatteri, Masiero \cite{GuM} for the maximum principle, and refer to
Fuhrman, Masiero, Tessitore \cite{FMT}, Gozzi, Masiero \cite{GoM1},\ \cite{GoM2} for the dynamic programming principle
and the probability interpretation of related Hamilton-Jacobi-Bellmam equations. As we know,
when adopting dual method to investigate the  necessary conditions of optimality for control systems with delay,
 the dual equations of the first-order variational equations are anticipated backward stochastic differential equations (BSDEs), which were first considered by Peng and Yang \cite{PY}
in 2009. Later, many scholars dedicated themselves to studying this kind of equations, such as Yang, Elliot \cite{YE},
Lu,\ Ren \cite{LR}.

On the other hand, mean-field BSDEs as the limit state of characterizing the asymptotic behavior of large stochastic particle systems
with mean-field interaction when the size of the system becomes very large, have also received a lot of attentions, for example,
Buckdahn, Djehiche, Li, Peng \cite{BDLP}, Buckdahn, Li, Peng \cite{BLP}. In particular, a recent series of works of  Lions \cite{Lions}
(or see the note edited by Cardaliaguet \cite{Car}) gave a huge impulse to investigate the general mean-field BSDEs, that is,
the coefficients depending on the law of the solution, not the expectation, see Buckdahn, Li, Peng, Rainer \cite{BLPR}, Hao, Li \cite{HL},
Li \cite{Li}, Agram \cite{A} and so on.

In this paper we are interested in the following general anticipated  mean-field BSDE with jumps:
\begin{equation}\label{equ 3.1}
\left\{
   \begin{aligned}
    -dY_t&=f(t,Y_t,Z_t,\int_G K_t(e)l(e)\lambda(de),A_t,B_t,C_t,\overline{A}_t,\overline{B}_t,\overline{C}_t,P_{\Pi_t})dt\\
         &\quad-Z_tdW_t-\int_GK_t(e)N_\lambda(de,dt),\\
      Y_t&=\varphi_t,\ Z_t=\phi_t,\ K_t=\psi_t(\cdot),\ t\in[T,T+M],
    \end{aligned}
   \right.
   \end{equation}
where
$$
\begin{aligned}
A_t&=Y_{t+\delta_1(t)},\ B_t=Z_{t+\delta_2(t)},\
C_t=\int_GK_{t+\delta_3(t)}(e)l(e)\lambda(de),\\
\overline{A}_t&=\int_0^{\delta_1(t)}e^{-\rho s}Y_{t+s}ds,\
\overline{B}_t=\int_0^{\delta_2(t)}e^{-\rho s}Z_{t+s}ds,\\
\overline{C}_t&=\int_0^{\delta_3(t)} e^{-\rho s}\int_GK_{t+s}(e)l(e)\lambda(de)ds,\ \rho>0,\
 \Pi_t=(Y_t,Z_t,\int_GK_t(e)l(e)\lambda(de)),
\end{aligned}
$$
$\varphi,\ \phi,\ \psi$ are given functions on $[T,T+M]$,
$W$ is a $d$-dimensional Brownian motion, $N_\lambda$ is a Poisson martingale measure, $P_\xi=P\circ \xi^{-1}$
is the law (or, called distribution) of  random variable $\xi\in L^2(\Omega,\mathcal{F},P)$, the mapping $l(\cdot):G\rightarrow\mathbb{R}$
satisfies for some given constant $C>0$, $0<l(e)\leq C(1\wedge|e|)$.
Here $A_t$ and $\overline{A}_t$ can be
regarded as the counterparts of the one point delay and the average delay in the corresponding mean-field SDDE with jumps.
We may call them the one point anticipated term and the average anticipated term, respectively.
$B_t,\overline{B}_t,C_t, \overline{C}_t$ can be understood similarly.
We proved the existence and uniqueness theorem, comparison theorem for one-dimensional anticipated mean-field BSDEs with jumps.

The motivation comes on the one hand from the rapid development of the theory of mean-field BSDEs, on the other hand
from the necessary of  studying the optimal control problems driven by mean-field SDDEs with jumps or anticipated mean-field BSDEs with jumps.

Compared with the spermic work of Peng and Yang \cite{PY}, the potential obstacle of this paper lies in involving the mean-field term and jump term, which means that we need more subtle calculation, see the proof of Lemma \ref{le 4.1}.

This paper is organized as follows. In Section 2, we recall the notion of the derivative in the Wasserstein space
and some usual functional spaces. Section 3
is devoted to showing the existence and uniqueness theorem. Comparison theorem is supplied in Section 4.

\section{{\protect \large {Preliminaries}}}

The differentiability of a function defined on $\mathcal{P}_2(\mathbb{R}^d)$ and some usual spaces are introduced in this section.

\subsection{ Derivative in the Wasserstein Space}
Let $\mathcal{P}_2(\mathbb{R}^d)$ be the space of all the probability measures on $(\mathbb{R}^d,\mathcal{B}(\mathbb{R}^d))$
with finite second-order moment, which is endowed with the 2-Wasserstein's distance $W_2$:
$$
W_2(\nu_1,\nu_2):=\inf\Big\{
\big(\int_{\mathbb{R}^{2d}}|a_1-a_2|^2\pi(a_1,a_2)\big)^\frac{1}{2},\ \pi\in\mathcal{P}_2(\mathbb{R}^{2d})\
\text{with\ marginals}\ \nu_1\ \text{and}\ \nu_2\Big\}.
$$
By $<\cdot,\cdot>$ we denote the ``dual product" on $L^2(\mathcal{F};\mathbb{R}^d)$,  and by
$\delta_\theta$ the Direc measure at $\theta$.

Let us now recall the notion of the differentiability of a function $\varphi:\mathcal{P}_2(\mathbb{R}^d)\rightarrow \mathbb{R}$
in $\nu\in\mathcal{P}_2(\mathbb{R}^d)$. Let $(\Omega,\mathcal{F},P)$ be a ``rich enough" space,
i.e., for each $\nu\in\mathcal{P}_2(\mathbb{R}^d)$, there exists a random variable $\xi\in L^2(\mathcal{F};\mathbb{R}^d)$ such
that $\nu=P_\xi$.
\begin{definition}(see \cite{Lions})
For $\xi_0\in L^2(\mathcal{F};\mathbb{R}^d)$, we call the function $\varphi:\mathcal{P}_2(\mathbb{R}^d)\rightarrow \mathbb{R}$
is differentiable at $P_{\xi_0}$, if the ``lifted" function $\varphi^{\#}(\xi):=\varphi(P_\xi),\ \forall\xi\in L^2(\mathcal{F};\mathbb{R}^d)$
is differentiable at $\xi_0$ in Fr\'{e}chet sense.
\end{definition}
This means that there exists a continuous linear mapping $D\varphi^{\#}(\xi_0):L^2(\mathcal{F};\mathbb{R}^d)\rightarrow\mathbb{R}$
such that for $\zeta\in L^2(\mathcal{F};\mathbb{R}^d)$,
\begin{equation}\label{equ 2.1}
\varphi^{\#}(\xi_0+\zeta)-\varphi^{\#}(\xi_0)=D\varphi^{\#}(\xi_0)(\zeta)+o(||\zeta||_{L^2}).
\end{equation}
Riesz' Representation Theorem allows to show that there exists an $\eta_0\in L^2(\mathcal{F};\mathbb{R}^d)$ such that
 $$ D\varphi^{\#}(\xi_0)(\zeta)=<\eta_0,\zeta>. $$
In \cite{Car} Cardaliaguet proved the existence of a Borel measure function $h$ depending on the law of $\xi_0$, not on the random
variable $\xi_0$ itself, such that $\eta_0=h(\xi_0)$. Consequently, (\ref{equ 2.1}) can be
described as
\begin{equation}\label{equ 2.2}
\varphi(P_{\xi_0+\zeta})-\varphi(P_{\xi_0})=<h(\xi_0),\zeta>+o(||\zeta||_{L^2}),\ \zeta\in L^2(\mathcal{F};\mathbb{R}^d).
\end{equation}
We call $\partial_\nu\varphi(P_{\xi_0};a)=h(a),\ a\in \mathbb{R}^d$, the derivative of $\varphi:\mathcal{P}_2(\mathbb{R}^d)\rightarrow\mathbb{R}$
at $P_{\xi_0}$. It is
easy to see $D\varphi^{\#}(\xi_0)=h(\xi_0)=\partial_\nu\varphi(P_{\xi_0};\xi_0)$.

In this paper for convenience we assume all the functions $\varphi^{\#}:L^2(\mathcal{F};\mathbb{R}^d)\rightarrow\mathbb{R}$
are Fr\'{e}chet differential over the whole space $L^2(\mathcal{F};\mathbb{R}^d)$, which naturally guarantees the corresponding functions
$\varphi:\mathcal{P}_2(\mathbb{R}^d)\rightarrow \mathbb{R}$ are differentiable in all probability measures of
$\mathcal{P}_2(\mathbb{R}^d)$. Note that in this situation
$\partial_\nu\varphi(P_\xi;a),\ \xi\in L^2(\mathcal{F};\mathbb{R}^d),\ a\in\mathbb{R}^d$ is
$P_\xi(da)$-a.e., and moreover, from Lemma 3.2 in \cite{CD}, if there exists a constant $K>0$ such that
for $\xi_1,\xi_2\in L^2(\mathcal{F};\mathbb{R}^d)$,
$$
E|\partial_\nu\varphi(P_{\xi_1};\xi_1)-\partial_\nu\varphi(P_{\xi_2};\xi_2)|^2\leq K^2 E|\xi_1-\xi_2|^2,
$$
then for all $\xi\in L^2(\mathcal{F};\mathbb{R}^d)$, there is a $P_\xi$-version of
$\partial_\nu\varphi(P_\xi;\cdot):\mathbb{R}^d\rightarrow\mathbb{R}^d$ such that
$$
|\partial_\nu \varphi(P_\xi;a)-\partial_\nu \varphi(P_\xi;a')|\leq K|a-a'|,\ \text{for}\ a,a'\in\mathbb{R}^d.
$$

\subsection{Functional spaces}

Let $T$ be a given time horizon and let $(\Omega, \mathcal{F},P;\mathbb{F}=(\mathcal{F}_t)_{t\in[0,T]})$
be a completed filtered probability space on which is defined a $d$-dimensional Brownian motion $W$. Let $G\subseteq \mathbb{R}$
be a nonempty open set, equipped with its Borel $\sigma$-algebra $\mathcal{B}(G)$, and $\lambda$ be a $\sigma$-finite L\'{e}vy
measure on $(G,\mathcal{B}(G))$, i.e., $\lambda$ satisfying $\int_G(1\wedge|e|)\lambda(de)<+\infty$. Let $N$ be a
Poisson random measure on $[0,T]\times G$ independent of the Brownian motion $W$, with compensator $\mu(de,dt)=\lambda(de)dt$ such
that $\big\{N_\lambda((s,t]\times B)=(N-\mu)((s,t]\times B),\ {s\leq t}, \
B\in\mathcal{B}(G)$ \text{with} $\lambda(B)<+\infty \big\}$ is a martingale.
By $\mathcal{P}^o$ we denote the $\sigma$-field of $\mathbb{F}$-predictable subsets of $\Omega\times [0,T]$.

$\mathbb{F}=\{\mathcal{F}_t\}_{t\in[0,T]}$ is the natural filtration generated by the Brownian motion $W$ and Poisson random measure $N$, augmented with an
independent $\sigma$-algebra $\mathcal{G}^o\subset\mathcal{F}$, that is,
$$
\begin{aligned}
\stackrel{o}{\mathcal{F}}_t&=\sigma\big\{W_s,\ N([0,s]\times A)\big|\ s\leq t,
A\in\mathcal{B}(G)\big\},\\
\mathcal{F}_t:&=\bigcap\limits_{s:s>t}\stackrel{o}{\mathcal{F}}_s\vee\mathcal{G}^o\vee\mathcal{N}_P,\ t\in[0,T],\\
\end{aligned}
$$
where $\mathcal{N}_P$ is the set of all $P$-null subsets, and
$\mathcal{G}^o\subset\mathcal{F}$ has the following properties:\\
\indent{\rm i)} $\mathcal{G}_0$ is independent of
the Brownian motion $W$ and the Poisson random measure $N$;\\
\indent{\rm ii)} $\mathcal{G}_0$ is ``rich enough", i.e.,
$\mathcal{P}_2(\mathbb{R}^d)=\{ P_\xi,\ \xi\in L^2(\mathcal{G}_0;\mathbb{R}^d)\}.$

\indent The following space is used frequently.

\indent $S^2_{\mathbb{F}}(s,t;\mathbb{R}^d):=
\Big\{ \psi| \psi:\Omega\times[s,t]\rightarrow\mathbb{R}^d$ is an $\mathbb{F}$-adapted c\`{a}dl\`{a}g process with\\
\mbox{}\qquad \qquad\qquad\qquad\qquad\qquad\qquad
$E[\mathop{\rm sup}\limits_{s\leq r\leq t}|\psi_r|^2]<+\infty\Big\}.$

\indent $\mathcal{H}^2_{\mathbb{F}}(s,t;\mathbb{R}^d):=
\Big\{\psi| \psi:\Omega\times[s,t]\rightarrow\mathbb{R}^d$ is an $\mathbb{F}$-predictable
process with\\
\mbox{}\qquad \qquad\qquad\qquad\qquad\qquad\qquad $||\psi||^2:=E[\int_s^t|\psi_r|^2dr]<+\infty\Big\}$.

\indent $\mathcal{K}^2_{\lambda}(s,t;\mathbb{R}^d):=
\Big\{K| K:\Omega\times[s,t]\times G\rightarrow\mathbb{R}^d$ is  $\mathcal{P}^o\otimes\mathcal{B}(G)$-measurable with\\
\mbox{}\qquad \qquad\qquad\qquad\qquad\qquad\qquad $||K||^2:=E[\int_s^t\int_G|K_r(e)|^2\lambda(de)dr]<+\infty\Big\}$.

In what follows, by $\mathcal{S}_{\mathbb{F}}^2(s,T+M), \mathcal{H}_{\mathbb{F}}^2(s,T+M), K_\lambda^2(s,T+M)$ we denote $\mathcal{S}_{\mathbb{F}}^2(s,T+M;\mathbb{R}), \mathcal{H}_{\mathbb{F}}^2(s,T+M;\mathbb{R}),
K_\lambda^2(s,T+M;\mathbb{R}),$ for short.

\section{{\protect \large {Existence and uniqueness theorem}}}

In this section we shall show the existence and uniqueness theorem of the equation (\ref{equ 3.1}).

Let $\delta_i:[0,T]\rightarrow \mathbb{R}^+,\ i=1,2,3,$ satisfy:\\
\textbf{(H3.1)} (i) There exists a constant $M>0$, such that for $t\in[0,T],$\\
\begin{equation}\label{equ 3.2}
t+h(t)\leq T+M, \ h=\delta_1,\delta_2,\delta_3;
\end{equation}
(ii) There exists a constant $L>0$, such that for $t\in[0,T]$ and
for all nonnegative and integrable functions $\Gamma(\cdot), \overline{\Gamma}(\cdot,\cdot),$
\begin{equation}\label{equ 3.3}
\begin{aligned}
&\int_t^T\Gamma(s+\delta_i(s))ds\leq L\int_t^{T+M}\Gamma(s)ds, \quad i=1,2,\\
&\int_t^T\int_G\overline{\Gamma}(s+\delta_3(s),e)\lambda(de)ds\leq L\int_t^{T+M}\int_G\overline{\Gamma}(s,e)\lambda(de)ds.
\end{aligned}
\end{equation}
\\

Let the mapping
$$
\begin{aligned}
&f(\omega,s,y,z,k,\xi,\eta,\zeta,\overline{\xi},\overline{\eta},\overline{\zeta},\nu):\Omega\times[0,T]\times\mathbb{R}
\times\mathbb{R}^d\times \mathbb{R}\times\Big(L^2(\mathcal{F}_{r_1};\mathbb{R})\times L^2(\mathcal{F}_{r_2};\mathbb{R}^d)\times\\
&\qquad\qquad\qquad\qquad\qquad \qquad\qquad\qquad
 L^2(\mathcal{F}_{r_3};\mathbb{R})\Big)^2\times \mathcal{P}_2(\mathbb{R}^{(1+d+1)})
 \rightarrow L^2(\mathcal{F}_s;\mathbb{R}), \\
\end{aligned}
$$
$r_1, r_2,r_3\in[s,T+M],$ satisfy:

\textbf{(H3.2)} (i) There exists a constant $C>0$, such that for all $s\in[0,T], y,y'\in\mathbb{R}, z,z'\in\mathbb{R}^d,$
$k,k'\in \mathbb{R},\ \xi_\cdot, \overline{\xi}_\cdot, \xi'_\cdot, \overline{\xi}'_\cdot\in
\mathcal{H}_{\mathbb{F}}^2(s,T+M),$
$\eta_\cdot, \overline{\eta}_\cdot, \eta'_\cdot, \overline{\eta}'_\cdot\in
\mathcal{H}_{\mathbb{F}}^2(s,T+M;\mathbb{R}^d),$\ $\zeta_\cdot, \overline{\zeta}_\cdot, \zeta'_\cdot, \overline{\zeta}'_\cdot\in
\mathcal{H}_{\mathbb{F}}^2(s,T+M),$\ $r_1,r_2,r_3\in[s,T+M],$ $\nu,\nu'\in \mathcal{P}_2(\mathbb{R}^{1+d+1}),$ $P$-a.s.,
$$
\begin{aligned}
&|f(s,y,z,k,\xi_{r_1},\eta_{r_2},\zeta_{r_3},\overline{\xi}_{r_1},\overline{\eta}_{r_2},\overline{\zeta}_{r_3},\nu)
-f(s,y',z',k',\xi'_{r_1},\eta'_{r_2},\zeta'_{r_3},\overline{\xi}'_{r_1},\overline{\eta}'_{r_2},\overline{\zeta}'_{r_3},\nu')|\\
&\leq C\Big(|y-y'|+|z-z'|+|k-k'|
+E^{\mathcal{F}_s}\Big[|\xi_{r_1}-\xi'_{r_1}|+|\eta_{r_2}-\eta'_{r_2}|+| \zeta_{r_3}-\zeta'_{r_3}|\\
&\quad+|\overline{\xi}_{r_1}-\overline{\xi}'_{r_1}|+|\overline{\eta}_{r_2}-\overline{\eta}'_{r_2}|+|\overline{\zeta}_{r_3}-\overline{\zeta}'_{r_3}|\Big]
+W_2(\nu,\nu')\Big);
\end{aligned}
$$
(ii) $E\Big[\int_0^T|f(s,0,0,0,0,0,0,0,0,0,\delta_{\textbf{0}})|^2ds\Big]<+\infty$, where  $\delta_{\textbf{0}}$
denotes the Dirac measure at $(1+d+1)$-dimensional zero vector;\\
(iii) There exists a constant $C>0$, such that the mapping $l(\cdot):G\rightarrow\mathbb{R}$ satisfies $0<l(e)\leq C(1\wedge|e|)$.
\begin{theorem}\label{th 3.1}
Let the assumption \textbf{(H3.1)} and \textbf{(H3.2)} hold true, and let $\varphi_\cdot\in \mathcal{S}_{\mathbb{F}}^2(T,T+M)$,
$\phi_\cdot\in\mathcal{H}_{\mathbb{F}}^2(T,T+M;\mathbb{R}^d),$
$\psi_\cdot\in K_\lambda^2(T,T+M),$ the
anticipated mean-field BSDE (\ref{equ 3.1}) possesses a unique adapted solution
$$
(Y_\cdot,Z_\cdot,K_\cdot)\in \mathcal{S}_{\mathbb{F}}^2(0,T+M)\times
\mathcal{H}_{\mathbb{F}}^2(0,T+M;\mathbb{R}^d)\times K_\lambda^2(0,T+M).
$$
\end{theorem}

\begin{proof}
For given $(y,z,k)\in\mathcal{H}_{\mathbb{F}}^2(0,T+M)\times \mathcal{H}_{\mathbb{F}}^2(0,T+M;\mathbb{R}^d)\times K_\lambda^2(0,T+M)$, we
define the following norm, for $\beta>0$,
$$
||(y_\cdot,z_\cdot,k_\cdot)||_{\beta}:=E\int_0^{T+M}e^{\beta s}\Big(|y_s|^2+|z_s|^2+\int_G|k_s(e)|^2\lambda(de)\Big)ds,
$$
under which the Contractive Mapping Theorem can be applied more expediently.\\
Let us consider the equation
\begin{equation}\label{equ 3.4}
\left\{
   \begin{aligned}
    -dY_t&=f(t,y_t,z_t,\int_Gk_t(e)l(e)\lambda(de),a_t,b_t,c_t,\overline{a}_t,\overline{b}_t,\overline{c}_t,P_{\pi_t})dt\\
         &\quad-Z_tdW_t-\int_GK_t(e)N_\lambda(de,dt),\ t\in[0,T],\\
      Y_t&=\varphi_t,\ Z_t=\phi_t,\ K_t=\psi_t(\cdot),\ t\in[T,T+M],
    \end{aligned}
   \right.
   \end{equation}
where $(a_t,b_t,c_t,\overline{a}_t,\overline{b}_t,\overline{c}_t, \pi_t)$ are defined similar to
$(A_t,B_t,C_t,\overline{A}_t,\overline{B}_t,\overline{C}_t, \Pi_t)$ in (\ref{equ 3.1}),
but with $(y,z,k)$ instead of $(Y,Z,K)$.\\
It is easy to check that the equation (\ref{equ 3.4}) exists a unique solution
$$
(Y_\cdot,Z_\cdot,K_\cdot)\in \mathcal{S}_{\mathbb{F}}^2(0,T+M)\times
\mathcal{H}_{\mathbb{F}}^2(0,T+M;\mathbb{R}^d)\times K_\lambda^2(0,T+M).
$$
From which we can define a mapping $\Phi:\mathcal{H}_{\mathbb{F}}^2(0,T+M)\times \mathcal{H}_{\mathbb{F}}^2(0,T+M;\mathbb{R}^d)\times K_\lambda^2(0,T+M)
\rightarrow\mathcal{H}_{\mathbb{F}}^2(0,T+M)\times \mathcal{H}_{\mathbb{F}}^2(0,T+M;\mathbb{R}^d)\times K_\lambda^2(0,T+M)$ such that
$$
\Phi(y_\cdot,z_\cdot,k_\cdot)=(Y_\cdot,Z_\cdot,K_\cdot).
$$

Let us now show that $\Phi$ is a strictly contractive mapping for some suitable $\beta>0$. For this end, let
$(y^i,z^i,k^i)\in\mathcal{H}_{\mathbb{F}}^2(0,T+M)\times\mathcal{H}_{\mathbb{F}}^2(0,T+M;\mathbb{R}^d)\times K_\lambda^2(0,T+M)$ and
$(Y^i,Z^i,K^i)=\Phi(y^i,z^i,k^i), i=1,2.$ By $\Delta Y$ we denote the difference of $Y^1$ and $Y^2$,
and $\Delta Z, \Delta K, \Delta y, \Delta z, \Delta k$ have the similar meaning.

Applying It\^{o}'s formula to $e^{\beta s}|\Delta Y_s|^2,$ we have
\begin{equation}\label{equ 3.5}
\begin{aligned}
&d e^{\beta s}|\Delta Y_s|^2\\
&=\beta e^{\beta s}|\Delta Y_s|^2ds+e^{\beta s}2\Delta Y_{s-}d\Delta Y_s+e^{\beta s}|\Delta Z_s|^2ds
+\int_Ge^{\beta s}|\Delta K_s(e)|^2\lambda(de)ds\\
&\qquad\qquad\qquad\qquad\qquad\qquad\qquad\qquad\qquad\qquad \qquad\qquad
+\int_Ge^{\beta s}|\Delta K_s(e)|^2N_\lambda(de,ds)\\
&=e^{\beta s}\Big(\beta|\Delta Y_s|^2+|\Delta Z_s|^2+\int_G|\Delta K_s(e)|^2\lambda(de)\Big)ds
     -e^{\beta s}2\Delta Y_s\Delta f(s)ds\\
&\quad+e^{\beta s}2 \Delta Y_s\Delta Z_sdW_s
+e^{\beta s}2\Delta Y_{s-}\int_G\Delta K_s(e)N_\lambda(de,ds)+\int_Ge^{\beta s}|\Delta K_s(e)|^2N_\lambda(de,ds),
\end{aligned}
\end{equation}
where
$$
\begin{aligned}
\Delta f(s)&=f(t,y_t^1,z_t^1,\int_Gk_t^1(e)l(e)\lambda(de),a_t^1,b_t^1,c_t^1,\overline{a}_t^1,\overline{b}_t^1,\overline{c}_t^1,P_{\pi_t^1})\\
           &\quad-f(t,y_t^2,z_t^2,\int_Gk_t^2(e)l(e)\lambda(de),a_t^2,b_t^2,c_t^2,\overline{a}_t^2,\overline{b}_t^2,\overline{c}_t^2,P_{\pi_t^2}),\\
\pi_t^i&=(y_t^i,z_t^i,\int_Gk_t^i(e)l(e)\lambda(de)),  i=1,2.
\end{aligned}
$$
Integrating from $t$ to $T$, and then taking conditional expectation, it follows
\begin{equation}\label{equ 3.6}
\begin{aligned}
&e^{\beta t}|\Delta Y_t|^2+E^{\mathcal{F}_t}\Big[\int_t^Te^{\beta s}(\beta|\Delta Y_s|^2+|\Delta Z_s|^2+\int_G|\Delta K_s(e)|^2\lambda(de))ds\Big]
=E^{\mathcal{F}_t}\Big[\int_t^Te^{\beta s}2\Delta Y_s\Delta f(s)ds\Big].
\end{aligned}
\end{equation}
In particular, as $t=0$,
\begin{equation}\label{equ 3.7}
\begin{aligned}
&E\Big[\int_0^Te^{\beta s}(\frac{\beta}{2}|\Delta Y_s|^2+|\Delta Z_s|^2+\int_G|\Delta K_s(e)|^2\lambda(de))ds\Big]\\
&\leq \frac{2C^2}{\beta}E\Big[\int_0^Te^{\beta s}
\Big(|\Delta y_s|+|\Delta z_s|
+|\int_G\Delta k_s(e)l(e)\lambda(de)|
+E^{\mathcal{F}_s}
\Big[|\Delta y_{s+\delta_1(s)}|+|\Delta z_{s+\delta_2(s)}|\\
&\quad+|\int_G\Delta k_{s+\delta_3(s)}(e)l(e)\lambda(de)|
+|\int_0^{\delta_1(s)}e^{-\rho r}\Delta y_{s+r}dr|+|\int_0^{\delta_2(s)}e^{-\rho r}\Delta z_{s+r}dr|\\
&\quad+ |\int_0^{\delta_3(s)}e^{-\rho r}\int_G\Delta  k_{s+r}(e)l(e)\lambda(de)dr\Big]
+W_2(P_{\pi_s^1},P_{\pi_s^2})\Big)^2ds\Big].
\end{aligned}
\end{equation}
From H\"{o}lder inequality and  the fact $W_2(P_\xi,P_\eta)\leq \{E|\xi-\eta|^2\}^\frac{1}{2}$, we have
\begin{equation}\label{equ 3.8}
\begin{aligned}
&E\Big[\int_0^Te^{\beta s}(\frac{\beta}{2}|\Delta Y_s|^2+|\Delta Z_s|^2+\int_G|\Delta K_s(e)|^2\lambda(de))ds\Big]\\
&\leq \frac{20C^2}{\beta}E\Big[\int_0^Te^{\beta s}
\Big(|\Delta y_s|^2+|\Delta z_s|^2
+|\int_G\Delta k_s(e)l(e)\lambda(de)|^2
+|\Delta y_{s+\delta_1(s)}|^2+|\Delta z_{s+\delta_2(s)}|^2\\
&\quad+|\int_G\Delta k_{s+\delta_3(s)}(e)l(e)\lambda(de)|^2
+|\int_0^{\delta_1(s)}e^{-\rho r}\Delta y_{s+r}dr|^2+|\int_0^{\delta_2(s)}e^{-\rho r}\Delta z_{s+r}dr|^2\\
&\quad+ |\int_0^{\delta_3(s)}e^{-\rho r}\int_G\Delta  k_{s+r}(e)l(e)\lambda(de)dr|^2
+E\Big[|\Delta y_s|^2+|\Delta z_s|^2+|\int_G\Delta k_s(e)l(e)\lambda(de)|^2\Big]
\Big)ds\Big].\\
\end{aligned}
\end{equation}
From (\ref{equ 3.3}) and H\"{o}lder inequality, it is clear that
\begin{equation}\label{equ 3.9}
\begin{aligned}
\int_0^Te^{\beta s}|\int_G\Delta k_{s+\delta_3(s)}(e)l(e)\lambda(de)|^2ds
&\leq\int_G|l(e)|^2\lambda(de)\cdot\int_0^T\int_Ge^{\beta s}|\Delta k_{s+\delta_3(s)}(e)|^2\lambda(de)ds\\
&\leq L\int_G|l(e)|^2\lambda(de)\int_0^{T+M}\int_Ge^{\beta s}|\Delta k_{s}(e)|^2\lambda(de)ds.
\end{aligned}
\end{equation}
Moreover, notice $\delta_3(s)\leq T+M$, from H\"{o}lder inequality, it yields
\begin{equation}\label{equ 3.10}
\begin{aligned}
&\int_0^Te^{\beta s}\Big|\int_0^{\delta_3(s)}e^{-\rho r}\int_G\Delta  k_{s+r}(e)l(e)\lambda(de)dr\Big|^2ds\\
&\leq\int_0^Te^{\beta s}\int_0^{\delta_3(s)}e^{-2\rho r}dr\int_0^{\delta_3(s)}(\int_G\Delta  k_{s+r}(e)l(e)\lambda(de))^2drds\\
&\leq\frac{1}{2\rho}(1-e^{-2\rho(T+M)})\int_G|l(e)|^2\lambda(de)\cdot
\int_0^T\int_0^{\delta_3(s)}\int_Ge^{\beta s}|\Delta k_{s+r}(e)|^2\lambda(de)drds.
\end{aligned}
\end{equation}
Denote $C_{T,M}=\frac{1}{2\rho}(1-e^{-2\rho(T+M)})$, due to $s\leq T\leq T+r,\ r>0$, then
\begin{equation}\label{equ 3.11}
\begin{aligned}
&\int_0^Te^{\beta s}\Big|\int_0^{\delta_3(s)}e^{-\rho r}\int_G\Delta  k_{s+r}(e)l(e)\lambda(de)dr\Big|^2ds\\
&\leq C_{T,M}\int_G|l(e)|^2\lambda(de)\int_0^T\int_0^{\delta_3(s)}\int_Ge^{\beta (T+r)}|\Delta k_{s+r}(e)|^2\lambda(de)drds.
\end{aligned}
\end{equation}
Let $u=s+r$, from  $\delta_3(s)\leq T+M$, we have
\begin{equation}\label{equ 3.12}
\begin{aligned}
&\int_0^Te^{\beta s}\Big|\int_0^{\delta_3(s)}e^{-\rho r}\int_G\Delta  k_{s+r}(e)l(e)\lambda(de)dr\Big|^2ds\\
&\leq C_{T,M}e^{\beta T}\int_G|l(e)|^2\lambda(de)\int_0^T\int_0^{\delta_3(s)}\int_Ge^{\beta (u-s)}|\Delta k_{u}(e)|^2\lambda(de)duds\\
&\leq C_{T,M}e^{\beta T}T\int_G|l(e)|^2\lambda(de)\int_0^{T+M}\int_Ge^{\beta u}|\Delta k_{u}(e)|^2\lambda(de)du.
\end{aligned}
\end{equation}
Hence, from (\ref{equ 3.8}), (\ref{equ 3.9}), (\ref{equ 3.12}) and utilizing the argument of calculating
(\ref{equ 3.9}) and (\ref{equ 3.12}), it follows

$$
\begin{aligned}
&E\Big[\int_0^{T+M}e^{\beta s}\big(\frac{\beta}{2}|\Delta Y_s|^2+|\Delta Z_s|^2+\int_G|\Delta K_s(e)|^2\lambda(de)\big)ds\Big]\\
&\leq\frac{20C^2}{\beta}\Big(2+(1+\int_G|l(e)|^2\lambda(de))(L+C_{T,M}e^{\beta T}T)\Big)
E\Big[\int_0^{T+M}e^{\beta s}\big(|\Delta y_s|^2+|\Delta z_s|^2\\
&\quad+\int_G|\Delta k_s(e)|^2\lambda(de)\big)ds\Big].
\end{aligned}
$$
Choosing $\beta=40C^2\Big(2+(1+\int_G|l(e)|^2\lambda(de))(L+C_{T,M}e^{\beta T}T)\Big)+2$, then we obtain
$$
\begin{aligned}
&E\Big[\int_0^{T+M}e^{\beta s}\big(|\Delta Y_s|^2+|\Delta Z_s|^2+\int_G|\Delta K_s(e)|^2\lambda(de)\big)ds\Big]\\
&\leq\frac{1}{2}
E\Big[\int_0^{T+M}e^{\beta s}\big(|\Delta y_s|^2+|\Delta z_s|^2+\int_G|\Delta k_s(e)|^2\lambda(de)\big)ds\Big],
\end{aligned}
$$
which means that $\Phi$ is a strictly contractive mapping. Hence, the equation (\ref{equ 3.1}) exists a unique
solution $(Y,Z,K)\in \mathcal{H}^2_{\mathbb{F}}(0,T+M)\times\mathcal{H}^2_{\mathbb{F}}(0,T+M; \mathbb{R}^d)\times
\mathcal{K}_\lambda^2(0,T+M).$ Moreover, observing (\ref{equ 3.6}), with the similar argument
and Burkholder-Davis-Gundy inequality
one can check $Y\in \mathcal{S}_{\mathbb{F}}^2(0,T+M).$ The proof is completed.
\end{proof}

\begin{proposition}
Let the assumptions \textbf{(H3.1)} and \textbf{(H3.2)} be in force, then there exists a constant
$L_0$ depending on $L, C$ and $T$, such that for
$(\varphi, \phi,\psi)\in \mathcal{S}^2_{\mathbb{F}}(T,T+M)\times\mathcal{H}^2_{\mathbb{F}}(T,T+M;\mathbb{R}^d)\times \mathcal{K}^2_\lambda(T,T+M)$,
 and $t\in[0,T]$,
\begin{equation}\label{equ 3.13}
 \begin{aligned}
 &E\Big[\sup\limits_{t\leq s\leq T}|Y_s|^2+\int_t^T|Z_s|^2ds+\int_t^T\int_G|K_s(e)|^2\lambda(de)ds\Big]\\
 &\leq L_0E\Big[|\varphi_T|^2+\int_T^{T+M}\big(|\varphi_s|^2+|\phi_s|^2+\int_G|\psi_s(e)|^2\lambda(de)\big)ds
 +\int_t^T|f(s,0,0,0,0,0,0,0,0,0,\delta_\mathbf{0})|^2ds\Big],
 \end{aligned}
\end{equation}
\end{proposition}
where $\delta_\mathbf{0}$ is given in (\textbf{H3.2}).
The proof is standard; refer to Proposition 4.4 in Peng and Yang \cite{PY}.

\section{{\protect \large {Comparison theorem}}}

Let us now analyze the comparison theorem of the equation (\ref{equ 3.1}). For one thing,
Peng and Yang \cite{PY} have stated with two examples that the comparison theorem of anticipated BSDEs  does not
hold true when the coefficient $f$ is decreasing in the anticipated term of $Y$, and $f$ depends on the anticipated
term of $Z$. For another, if the coefficient $f$ depends on the mean-field term of $Z$, or $f$ is decreasing with respect
to the mean-field term of $Y$, the comparison theorem of mean-field BSDEs also becomes invalid, see counter-example in Buckdahn,
Li and Peng \cite{BLP}. Therefore, we here just consider the comparison theorem of a class of anticipated mean-field BSDEs with jumps. Let
us introduce it in detail.\\

We assume that for $r_1\in[s,T+M],$
$$
\begin{aligned}
&f(\omega,s,y,z,k,\xi,\overline{\xi},\nu):\Omega\times[0,T]\times\mathbb{R}
\times\mathbb{R}^d\times \mathbb{R}\times L^2(\mathcal{F}_{r_1};\mathbb{R})\times L^2(\mathcal{F}_{r_1};\mathbb{R})\times \mathcal{P}_2(\mathbb{R})
 \rightarrow L^2(\mathcal{F}_s;\mathbb{R})
\end{aligned}
$$
satisfy (\textbf{H3.2}).

Let us consider the following anticipated mean-field BSDE:\\
\begin{equation}\label{equ 3.14}
\left\{
   \begin{aligned}
    -dY_t&=f(t,Y_t,Z_t,\int_GK_t(e)l(e)\lambda(de),A_t,\overline{A}_t,P_{Y_t})dt-Z_tdW_t-\int_GK_t(e)N_\lambda(de,dt),\\
      Y_t&=\varphi_t,\ Z_t=\phi_t,\ K_t=\psi_t(\cdot),\ t\in[T,T+M],
    \end{aligned}
   \right.
   \end{equation}
where $A_t, \overline{A}_t$ are given in (\ref{equ 3.1}).

In order to prove the comparison theorem of anticipated mean-field BSDE with jumps (\ref{equ 3.14}), let us first investigate the
comparison theorem of general mean-field BSDEs with jumps.

\begin{lemma}\label{le 4.1}
Let $f_i:\Omega\times[0,T]\times\mathbb{R}\times\mathbb{R}^d\times\mathbb{R}\times\mathcal{P}_2(\mathbb{R})\rightarrow\mathbb{R},\ i=1,2$ be
the drivers and moreover, we assume that
there exists a constant $C>0$, such that the derivatives of $f_1$ with respect to $\nu$ and $k$ are positive and bounded by $C>0,$
i.e.,  $0<\partial_k f_1\leq C$ and $0<\partial_\nu f_1\leq C$.
Let $(Y^i,Z^i,K^i),\ i=1,2$ be the solution of the following mean-field BSDE with jumps:
\begin{equation}\label{equ 6.1}
\left\{
   \begin{aligned}
    -dY^i_t&=f_i(t,Y^i_t,Z^i_t,\int_GK^i_t(e)l(e)\lambda(de),P_{Y_t^i})dt-Z^i_tdW_t-\int_GK^i_t(e)N_\lambda(de,dt),\\
      Y^i_T&=\varphi^i_T.
    \end{aligned}
   \right.
   \end{equation}
If $f_1\geq f_2$ and $\varphi_T^1\geq\varphi_T^2$, then $Y^1_t\geq Y^2_t,$ a.e., a.s.
\end{lemma}
\begin{proof}
We denote $\Delta Y=Y^2-Y^1, \Delta Z=Z^2-Z^1,\Delta K=K^2-K^1,\Delta \varphi=\varphi^2-\varphi^1$, then
\begin{equation}\label{equ 6.2}
\left\{
   \begin{aligned}
    -d\Delta Y_s&=\big(\delta f(s)+\alpha_y(s) \Delta Y_s+\alpha_z(s) \Delta Z_s+\int_G \alpha_k(s)\Delta K_sl(e)\lambda(de)+
    \widehat{E}[\widehat{\alpha}_\nu(s)\widehat{\Delta Y}_s]
    \big)ds-\Delta Z_sdW_s\\
                &\quad-\int_G\Delta K_s(e)N_\lambda(de,ds),\ s\in[0,T],\\
      \Delta Y_T&=\Delta\varphi_T,
    \end{aligned}
   \right.
   \end{equation}
where  for $\ell=y,z,k$,
$$
\begin{aligned}
\delta f(s)&:=f_2(s,Y^2_s,Z^2_s,\int_GK^2_s(e)l(e)\lambda(de), P_{Y^2_s})-f_1(s,Y^2_s,Z^2_s,\int_GK^2_s(e)l(e)\lambda(de), P_{Y^2_s}),\\
\alpha_\ell(s)&:=\int_0^1\frac{\partial f_1}{\partial \ell}
(s,Y^1_s+\rho(Y^2_s-Y^1_s),Z^1_s+\rho(Z^2_s-Z^1_s),\int_G(K_s^1(e)+\rho(K_s^2(e)-K_s^1(e))l(e)\lambda(de),\\
&\qquad\qquad\qquad\qquad\qquad\qquad\qquad\qquad\qquad\qquad\qquad\qquad \qquad \qquad \qquad\qquad   P_{Y^1_s+\rho(Y^2_s-Y^1_s)})d\rho,\\
\widehat{\alpha}_\nu(s)&:=\int_0^1\frac{\partial f_1}{\partial \nu}
(s,Y^1_s+\rho(Y^2_s-Y^1_s),Z^1_s+\rho(Z^2_s-Z^1_s),\int_G(K_s^1(e)+\rho(K_s^2(e)-K_s^1(e)) l(e)\lambda(de),\\
&\qquad\qquad\qquad\qquad\qquad\qquad\qquad\qquad\qquad\qquad\qquad \qquad
 P_{Y^1_s+\rho(Y^2_s-Y^1_s)};
\widehat{Y}^1_s+\rho(\widehat{Y}^2_s-\widehat{Y}^1_s))d\rho.\\
\end{aligned}
$$
Obviously, from (\textbf{H3.2})  it follows $|\alpha_\ell(s)|\leq C, s\in[0,T].$\\
Applying It\^{o}'s formula to $((\Delta Y_t)^+)^2$  we obtain
\begin{equation}\label{equ 3.17}
\begin{aligned}
&((\Delta Y_t)^+)^2+\int_t^T\mathbbm{1}_{\{\Delta Y_s>0\}}|\Delta Z_s|^2ds
+\int_t^T\int_G\Big(((\Delta Y_{s-}+\Delta K_s(e))^+)^2-((\Delta Y_{s-})^+)^2\\
&\qquad\qquad\qquad\qquad\qquad\qquad\qquad\qquad\qquad\qquad\qquad
  -2\mathbbm{1}_{\{\Delta Y_s>0\}}\Delta Y_{s-}\Delta K_s(e)\Big)N(de,ds)\\
&=((\Delta Y_T)^+)^2+\int_t^T2\mathbbm{1}_{\{\Delta Y_s>0\}}\Delta Y_s
\Big\{\alpha_y(s)\Delta Y_s+\alpha_z(s)\Delta Z_s+\int_G\alpha_k(s)\Delta K_s(e)l(e)\lambda(de)\\
&\quad+ \widehat{E}[\widehat{\alpha}_\nu(s)\widehat{\Delta Y}_s]+\delta f(s)\Big\}ds
-\int_t^T2\mathbbm{1}_{\{\Delta Y_s>0\}} \Delta Y_s\Delta Z_s dW_s\\
&\quad-\int_t^T\int_G2\mathbbm{1}_{\{\Delta Y_s>0\}}\Delta Y_s \Delta K_s(e)N_\lambda(de,ds).
\end{aligned}
\end{equation}
Taking expectation on both sides of (\ref{equ 3.17}) and notice $\delta f(s)\leq0, \Delta Y_T\leq0$, one has
\begin{equation}\label{equ 3.18}
\begin{aligned}
&E\Big[((\Delta Y_t)^+)^2+\int_t^T\mathbbm{1}_{\{\Delta Y_s>0\}}|\Delta Z_s|^2ds
+\int_t^T\int_G\Big(((\Delta Y_{s}+\Delta K_s(e))^+)^2-((\Delta Y_{s})^+)^2\\
&\qquad\qquad\qquad\qquad\qquad\qquad\qquad\qquad\qquad\qquad\qquad\qquad\qquad
  -2(\Delta Y_{s})^+\Delta K_s(e)\Big)\lambda(de)ds\Big]\\
&\leq  E\Big[\int_t^T2\mathbbm{1}_{\{\Delta Y_s>0\}}\Delta Y_s
\Big\{\alpha_y(s)\Delta Y_s+\alpha_z(s)\Delta Z_s+\int_G\alpha_k(s)\Delta K_s(e)l(e)\lambda(de)
+ \widehat{E}[\widehat{\alpha}_\nu(s)\widehat{\Delta Y}_s]\Big\}ds\Big].
\end{aligned}
\end{equation}
On the other hand, from the boundness of $\alpha_z(s)$ and the basic inequality $2ab\leq 2a^2+\frac{1}{2}b^2$,  it follows
\begin{equation}\label{equ 3.19}
E\Big[\int_t^T2\mathbbm{1}_{\{\Delta Y_s>0\}}\Delta Y_s\alpha_z(s)\Delta Z_sds\Big]\\
\leq 2C^2E[\int_t^T((\Delta Y_s)^+)^2ds]+\frac{1}{2}E[\int_t^T\mathbbm{1}_{\{\Delta Y_s>0\}}|\Delta Z_s|^2ds ]
\end{equation}
and  moreover, from Jensen's inequality and the assumption $0<\partial_\nu f_1\leq C$ we obtain
\begin{equation}\label{equ 3.20}
\begin{aligned}
 &E\Big[\int_t^T2\mathbbm{1}_{\{\Delta Y_s>0\}}\Delta Y_s\widehat{E}[\widehat{\alpha}_\nu(s)\widehat{\Delta Y}_s]ds\Big]\\
 &\leq 2CE\int_t^T \mathbbm{1}_{\{\Delta Y_s>0\}}\Delta Y_sE[(\Delta Y _s)^+]
 \leq 2C E[\int_t^T((\Delta Y_s)^+)^2ds].
 \end{aligned}
\end{equation}
Combining (\ref{equ 3.18}), (\ref{equ 3.19}) and (\ref{equ 3.20}),  (\ref{equ 3.18}) can be rewritten as
\begin{equation}\label{equ 3.21}
\begin{aligned}
&E\Big[((\Delta Y_t)^+)^2+\frac{1}{2}\int_t^T\mathbbm{1}_{\{\Delta Y_s>0\}}|\Delta Z_s|^2ds
+\int_t^T\int_G\Big(((\Delta Y_{s}+\Delta K_s(e))^+)^2-((\Delta Y_{s})^+)^2\\
&\qquad\qquad\qquad\qquad\qquad\qquad\qquad\qquad\qquad\qquad\qquad\qquad\qquad
  -2(\Delta Y_{s})^+\Delta K_s(e)\Big)\lambda(de)ds\Big]\\
&\leq  (4C+2C^2)E[\int_t^T ((\Delta Y_{s})^+)^2 ds ]+
E\Big[\int_t^T\int_G2\mathbbm{1}_{\{\Delta Y_s>0\}}\Delta Y_s\Delta K_s(e)\alpha_k(s)l(e)\lambda(de)ds\Big].
\end{aligned}
\end{equation}
For convenience, we
denote $A=\{(s,\omega)|\Delta Y_s>0\}$ and $B=\{(s,\omega)|\Delta Y_{s}+\Delta K_s(e)>0\}$.
It is easy to check
\begin{equation}\label{equ 3.22}
\begin{aligned}
&E\Big[\int_t^T\int_G\Big(((\Delta Y_{s}+\Delta K_s(e))^+)^2-((\Delta Y_{s})^+)^2
  -2(\Delta Y_{s})^+\Delta K_s(e)\Big)\lambda(de)ds\Big]\\
\end{aligned}
\end{equation}
 $$
\begin{aligned}
&=E\Big[\int_t^T\int_G(\mathbbm{1}_{AB}+\mathbbm{1}_{A^cB})(\Delta Y_{s}+\Delta K_s(e))^2
-(\mathbbm{1}_{AB}+\mathbbm{1}_{AB^c})\Big((\Delta Y_{s})^2+2\Delta Y_{s}\Delta K_s(e)\Big)\lambda(de)ds\\
&\geq E\Big[\int_t^T\int_G\mathbbm{1}_{AB}|\Delta K_s(e)|^2-\mathbbm{1}_{AB^c}((\Delta Y_{s})^2+2\Delta Y_{s}\Delta K_s(e))
\lambda(de)ds.
\end{aligned}
$$
Combining  (\ref{equ 3.21}) with (\ref{equ 3.22}), it follows
\begin{equation}\label{equ 3.23}
\begin{aligned}
&E[((\Delta Y_t)^+)^2]+\frac{1}{2}E\Big[\int_t^T\mathbbm{1}_{A}|\Delta Z_s|^2ds\Big]
+E\Big[\int_t^T\int_G\mathbbm{1}_{AB}|\Delta K_s(e)|^2\lambda(de)ds\Big]\\
&\quad+E\Big[\int_t^T\int_G\mathbbm{1}_{AB^c}\Big(-(\Delta Y_{s})^2-2\Delta Y_{s}\Delta K_s(e)(1+\alpha_k(s)l(e))\Big)\lambda(de)ds\Big]\\
&\leq  (4C+2C^2)E[\int_t^T ((\Delta Y_{s})^+)^2 ds]
+E\Big[\int_t^T\int_G2\mathbbm{1}_{AB}\Delta Y_s\Delta K_s(e)\alpha_k(s)l(e)\lambda(de)ds\Big].
\end{aligned}
\end{equation}
According to the boundness assumption of $\alpha_k(s)$, H\"{o}lder inequality, the inequality $2ab\leq 2a^2+\frac{1}{2}b^2$,
we have
\begin{equation}\label{equ 3.24}
\begin{aligned}
&E[((\Delta Y_t)^+)^2]+\frac{1}{2}E[\int_t^T\mathbbm{1}_{A}|\Delta Z_s|^2ds]
+\frac{1}{2}E\Big[\int_t^T\int_G\mathbbm{1}_{AB}|\Delta K_s(e)|^2\lambda(de)ds\Big]\\
&+E\Big[\int_t^T\int_G\mathbbm{1}_{AB^c}\Big(-(\Delta Y_{s})^2-2\Delta Y_{s}\Delta K_s(e)(1+\alpha_k(s)l(e))\Big)\lambda(de)ds\Big]\\
&\leq  \Big(4C+2C^2+2C^2\int_G|l(e)|^2\lambda(de)\Big)E[\int_t^T ((\Delta Y_{s})^+)^2 ds].
\end{aligned}
\end{equation}

We argue that
$\Gamma:=E[\int_t^T\int_G\mathbbm{1}_{AB^c}\Big(-(\Delta Y_{s})^2-2\Delta Y_{s}\Delta K_s(e)(1+\alpha_k(s)l(e))\Big)\lambda(de)ds]\geq0$.\\
In fact, for each $e\in G$ and  for any $(s,\omega)\in AB^c$, we have $0<\Delta Y_s\leq -\Delta K_s(e)$.
The nonnegative assumptions on $\partial_kf_1$ and $l(\cdot)$ can imply $\Gamma\geq0$ easily.
Hence,\begin{equation}\label{equ 3.25}
\begin{aligned}
&E[((\Delta Y_t)^+)^2]
&\leq  \Big(4C+2C^2+2C^2\int_G|l(e)|^2\lambda(de)\Big)E[\int_t^T ((\Delta Y_{s})^+)^2 ds].
\end{aligned}
\end{equation}
(\ref{equ 3.25}) and Gronwall lemma could show the desired result.
\end{proof}\\

Let us state the second main result of this paper--Comparison Theorem. We make an extra assumption:\\
\textbf{(H3.3)}
Let $f_i, i=1,2$ be two drivers of (\ref{equ 3.14}) and satisfy:\\
(i) $f_2(t,y,z,k,\xi_r,\overline{\xi}_r,\nu)\geq f_2(t,y,z,k,\xi'_r,\overline{\xi}'_r,\nu),$
$(t,y,z,k,\nu)\in[0,T]\times \mathbb{R}\times \mathbb{R}^d\times \mathbb{R}\times \mathcal{P}_2(\mathbb{R}),$
if $\xi_r\geq\xi'_r,\overline{\xi}_r\geq\overline{\xi}'_r,\quad \xi_r, \xi'_r,\overline{\xi}_r, \overline{\xi}'_r\in \mathcal{H}^2_{\mathbb{F}}(t,T+M);$\\
(ii) There exists a constant $C>0$, such that the derivatives of $f_1$ with respect to $\nu$ and $k$ are positive and bounded by $C>0,$
i.e.,  $0<\partial_k f_1\leq C$ and $0<\partial_\nu f_1\leq C$.

\begin{theorem}\label{th 3.2}(Comparison Theorem)
Let the assumptions \textbf{(H3.1)}, \textbf{(H3.2)} and \textbf{(H3.3)} hold true and  let $\varphi^i\in S^2_{\mathbb{F}}(T,T+K),\ i=1,2.$
By $(Y^i,Z^i,K^i)$ we denote the solution of the equation (\ref{equ 3.14}) with data $(f_i, \varphi^i)$. If
 $\varphi^1_s\geq\varphi^2_s, s\in[T,T+K]$ and
 $f_1(s,y,z,k,\theta_r,\overline{\theta}_r,\nu)\geq f_2(s,y,z,k,\theta_r,\overline{\theta}_r,\nu),$
 for $s\in[0,T],$ $y\in \mathbb{R},$ $z\in\mathbb{R}^d$, $k\in \mathbb{R}$,
 $\theta_r,\overline{\theta}_r\in \mathcal{H}^2_{\mathbb{F}}(s,T+M),$  $\nu\in\mathcal{P}_2(\mathbb{R}),$
 $r\in[t,T+M],$ then $Y^1_t\geq Y^2_t, \ a.s., a.e.$
\end{theorem}

\begin{proof}
For $i=1,2,3,\cdot\cdot\cdot$, we set $A_s^i=Y^i_{s+\delta_1(s)}, \overline{A}_s^i=\int_0^{\delta_1(s)}e^{-\rho u}Y^i_{s+u}du$.
Let $(Y^3_\cdot,Z^3_\cdot,K^3_\cdot)\in \mathcal{S}^2_{\mathbb{F}}(0,T)\times \mathcal{H}^2_{\mathbb{F}}(0,T;\mathbb{R}^d)
\times \mathcal{K}^2_\lambda(0,T)$ be the solution of the following  mean-filed BSDE with jumps:
\begin{equation}\label{equ 4.13}
\left\{
   \begin{aligned}
    Y^3_t&=\varphi_T^2+\int_t^Tf_2(s,Y^3_s,Z^3_s,\int_GK^3_s(e)l(e)\lambda(de),A^1_s,\overline{A}^1_s,P_{Y^3_s})ds
    -\int_t^TZ^3_sdW_s\\
    &\quad-\int_t^T\int_GK^3_s(e)N_\lambda(de,dt),\ t\in[0,T],\\
      Y^3_t&=\varphi^2_t,\ t\in[T,T+M].
    \end{aligned}
   \right.
   \end{equation}
From Lemma \ref{le 4.1}, it yields  $Y^1_t\geq Y^3_t,$ a.e., a.s.\\
We now consider
\begin{equation}\label{equ 4.14}
\left\{
   \begin{aligned}
    Y^4_t&=\varphi_T^2+\int_t^Tf_2(s,Y^4_s,Z^4_s,\int_GK^4_s(e)l(e)\lambda(de),A^3_s,\overline{A}^3_s,P_{Y^4_s})ds-\int_t^TZ^4_sdW_s\\
    &\quad-\int_t^T\int_GK^4_s(e)N_\lambda(de,dt),\ t\in[0,T],\\
      Y^3_t&=\varphi^2_t,\ t\in[T,T+M].
    \end{aligned}
   \right.
   \end{equation}
Since $f_2(s,y,z,k,\cdot,\cdot,\nu)$ is increasing, one can check  $Y^3_t\geq Y^4_t,$ a.e., a.s.
Repeating the above argument, we obtain  $Y^3_t\geq Y^n_t,$ a.e., a.s., where
$(Y^n_\cdot,Z^n_\cdot,K^n_\cdot)$ is the solution of mean-field BSDE with jumps: for $n\geq4$,
\begin{equation}\label{equ 4.15}
\left\{
   \begin{aligned}
    Y^n_t&=\varphi_T^2+\int_t^Tf_2(s,Y^n_s,Z^n_s,\int_GK^n_s(e)l(e)\lambda(de),A^{n-1}_s,\overline{A}^{n-1}_s,P_{Y^n_s})ds-\int_t^TZ^n_sdW_s\\
    &\quad-\int_t^T\int_GK^n_s(e)N_\lambda(de,dt),\ t\in[0,T],\\
      Y^n_t&=\varphi^2_t,\ t\in[T,T+M].
    \end{aligned}
   \right.
   \end{equation}
We shall prove the existence of the  limit of $(Y^n_\cdot,Z^n_\cdot,K^n_\cdot)$, and show that it is just
$(Y^2_\cdot,Z^2_\cdot,K^2_\cdot).$ For this purpose, we define for $n\geq1$,
$$
(\widetilde{Y}^n_\cdot,\widetilde{Z}^n_\cdot,\widetilde{K}^n_\cdot,\widetilde{A}^n_s,\widetilde{\bar{A}}^n_s)
:=(Y^n_\cdot-Y^{n-1}_\cdot,Z^n_\cdot-Z^{n-1}_\cdot,K^n_\cdot-K^{n-1}_\cdot,A^{n}_\cdot-A^{n-1}_\cdot,\overline{A}^{n}_\cdot-\overline{A}^{n-1}_\cdot).
$$
From the Lipschitz property of $f$ and the fact $W_2(P_\xi,P_\eta)\leq \{E|\xi-\eta|^2\}^\frac{1}{2},$
the It\^{o}'s formula allows to show, for $\beta>0,$
$$
\begin{aligned}
&E\Big[\int_0^Te^{\beta s}(\frac{\beta}{2}|\widetilde{Y}^n_s|^2+|\widetilde{Z}^n_s|^2+\int_G|\widetilde{K}^n_s(e)|^2\lambda(de))ds\Big]\\
&\leq \frac{2}{\beta}E\Big[
\int_0^Te^{\beta s}|f_2(s,Y^n_s,Z^n_s,\int_GK^n_s(e)l(e)\lambda(de),A^{n-1}_s,\overline{A}^{n-1}_s,P_{Y^n_s})\\
&\qquad\qquad\qquad-f_2(s,Y^{n-1}_s,Z^{n-1}_s,\int_GK^{n-1}_s(e)l(e)\lambda(de),A^{n-2}_s,\overline{A}^{n-2}_s,P_{Y^{n-1}_s})|^2ds\Big]\\
&\leq\frac{2C^2}{\beta}E\Big[\int_0^Te^{\beta s}\Big(
|\widetilde{Y}^n_s|+|\widetilde{Z}^n_s|+|\int_G\widetilde{K}^n_sl(e)\lambda(de)|+
E^{\mathcal{F}_s}[|\widetilde{A}^{n-1}_s|+|\widetilde{\bar{A}}^{n-1}_s|]+W_2(P_{Y^n_s}, P_{Y^{n-1}_s})
\Big)^2ds\Big]\\
&\leq\frac{12C^2}{\beta}E\Big[\int_0^T e^{\beta s}\Big(
|\widetilde{Y}^n_s|^2+|\widetilde{Z}^n_s|^2+|\int_G\widetilde{K}^n_sl(e)\lambda(de)|^2
+E^{\mathcal{F}_s}[|\widetilde{A}^{n-1}_s|^2+|\widetilde{\bar{A}}^{n-1}_s|^2]
+E[|\widetilde{Y}^n_s|^2]\Big)ds\Big].
\end{aligned}
$$
Recall the definitions of $\widetilde{A}^{n-1}_s,\ \widetilde{\bar{A}}^{n-1}_s$. On the one hand, one can check
$$
E\int_0^TE^{\mathcal{F}_s}[e^{\beta s}|\widetilde{Y}^{n-1}_{s+\delta_1(s)}|^2]ds\leq L E\int_0^Te^{\beta s}|\widetilde{Y}^{n-1}_{s}|^2ds.
$$
On the other hand, following the  similar argument of the proof of Theorem 3.1 (refer to (\ref{equ 3.10})-(\ref{equ 3.12})), it follows
$$
E\int_0^TE^{\mathcal{F}_s}[|\int_0^{\delta_1(s)}e^{-\rho u}\widetilde{Y}^{n-1}_{s+u}du|^2e^{\beta s}]ds
\leq \frac{1}{2\rho}(1-e^{-2\rho(T+M)})Te^{\beta T}E\int_0^Te^{\beta s}|\widetilde{Y}_s^{n-1}|^2ds.
$$
Consequently, we have
$$
\begin{aligned}
&E\Big[\int_0^Te^{\beta s}(\frac{\beta}{2}|\widetilde{Y}^n_s|^2+|\widetilde{Z}^n_s|^2+\int_G|\widetilde{K}^n_s(e)|^2\lambda(de))ds\Big]\\
&\leq\frac{12C^2}{\beta}\Big(2+\int_G(1\wedge|e|^2)\lambda(de)\Big)E\Big[\int_0^Te^{\beta s}(|\widetilde{Y}_s^{n}|^2ds+|\widetilde{Z}_s^{n}|^2
+\int_G|\widetilde{K}_s^{n}|^2\lambda(de))ds\Big]\\
&\quad+\frac{12C^2}{\beta}(L+\kappa_0)E\int_0^Te^{\beta s}|\widetilde{Y}_s^{n-1}|^2ds,
\end{aligned}
$$
where $\kappa_0=\frac{1}{2\rho}(1-e^{-2\rho(T+M)})Te^{\beta T}.$\\
Choosing $\beta=36C^2(2+\int_G(1\wedge|e|^2)\lambda(de)+L+\kappa_0)+3$, then
$$
\begin{aligned}
&E\Big[\int_0^Te^{\beta s}(|\widetilde{Y}^n_s|^2+|\widetilde{Z}^n_s|^2+\int_G|\widetilde{K}^n_s(e)|^2\lambda(de))ds\Big]\\
&\leq \frac{1}{2}E\Big[\int_0^Te^{\beta s}(|\widetilde{Y}^{n-1}_s|^2+|\widetilde{Z}^{n-1}_s|^2
+\int_G|\widetilde{K}^{n-1}_s(e)|^2\lambda(de))ds\Big].
\end{aligned}
$$
Therefore,
$$
\begin{aligned}
&E\Big[\int_0^Te^{\beta s}(|\widetilde{Y}^n_s|^2+|\widetilde{Z}^n_s|^2+\int_G|\widetilde{K}^n_s(e)|^2\lambda(de))ds\Big]\\
&\leq\frac{1}{2^{n-4}}E\Big[\int_0^Te^{\beta s}(|\widetilde{Y}^{4}_s|^2+|\widetilde{Z}^{4}_s|^2
+\int_G|\widetilde{K}^{4}_s(e)|^2\lambda(de))ds\Big],
\end{aligned}
$$
which implies that $(Y^n,Z^n,K^n)_{n\geq4}$ is a Cauchy sequence. By $(Y,Z,K)$ we denote
its limit. It is easy to get $(Y,Z)\in \mathcal{H}^2_{\mathbb{F}}(0,T;\mathbb{R}^{1+d})$
and $K\in \mathcal{K}^2_\lambda(0,T)$, and moreover, $Y^n_t\geq Y_t,$ a.e., a.s.\\
Taking limit in (\ref{equ 4.15}), one has
\begin{equation}\label{equ 4.16}
\left\{
   \begin{aligned}
    Y_t&=\varphi_T^2+\int_t^Tf_2(s,Y_s,Z_s,\int_GK_s(e)l_s(e)\lambda(de),A_s,\overline{A}_s,P_{Y_s})ds-\int_t^TZ_sdW_s\\
    &\quad-\int_t^T\int_GK_s(e)N_\lambda(de,dt),\ t\in[0,T],\\
      Y_t&=\varphi^2_t, t\in[T,T+M].
    \end{aligned}
   \right.
   \end{equation}
The existence and uniqueness of the solution of the anticipated mean-field BSDEs (see Theorem \ref{th 3.1})
allows to show $Y_t=Y^2_t,$ a.e., a.s. Then the desired result comes from the fact $Y^1_t\geq Y^3_t\geq\cdot\cdot\cdot\geq Y_t=Y^2_t,$
a.e., a.s.
\end{proof}

\renewcommand{\refname}{\large References}{\normalsize \ }

\end{document}